\documentclass[12pt]{amsart}
\usepackage{times}
\usepackage[T1]{fontenc}
\usepackage[latin1]{inputenc}
\usepackage{geometry}
\usepackage{amssymb}
\IfFileExists{url.sty}{\usepackage{url}}
                      {\newcommand{\url}{\texttt}}

\usepackage{comment}

\makeatletter
 \theoremstyle{plain}
 \newtheorem{thm}{Theorem}[section]
 \numberwithin{equation}{section} 
 \newtheorem{conj}{Conjecture}[section]
 \numberwithin{equation}{section}
 \numberwithin{figure}{section} 
 \theoremstyle{plain}
 \theoremstyle{definition}
 \newtheorem{defn}[thm]{Definition}
 \theoremstyle{plain}
 \newtheorem{lem}[thm]{Lemma} 
 \theoremstyle{remark}
 \newtheorem{rem}[thm]{Remark}
 \theoremstyle{plain}
 \theoremstyle{plain}
 \newtheorem{prop}[thm]{Proposition} 
 \theoremstyle{plain}
 \newtheorem{cor}[thm]{Corollary} 
 \theoremstyle{plain}
 
 \theoremstyle{definition}
 \newtheorem{example}[thm]{Example}

\newcommand{\card}[1]{\left| #1 \right|}
\DeclareMathOperator{\codim}{codim}

\def\R{\mathbb {R}}
\def\N{\mathbb {N}}
\def\Z{{\mathbb Z}}
\def\subgp{<}
\def\ideal{\lhd}
\def\isom{\simeq}

\def\ie{i.e.\ }

\newcommand\FF{\mathbb{F}}
\newcommand\eqdef{\overset{\text{def}}{=}}
\newcommand\CA{\mathcal{A}}
\newcommand\HH{\mathcal{H}}
\DeclareMathOperator\soc{M}
\newcommand\socG{\soc(G)}

\DeclareMathOperator\Core{Core}
\DeclareMathOperator\RCore{RC}
\newcommand\CG[1]{\Core_G(#1)}
\newcommand\RC[1]{\RCore_G(#1)}

\newcommand\MM{\mathcal{M}}
\newcommand\TT{\mathcal{T}}

\newcommand\MAGMA{\textsf{MAGMA}$\,\,$}
\makeatother

\begin{document}

\title{Minimal Permutation Representations of Nilpotent Groups}
\date{\today}

\author[Elias]{Ben Elias}
\address{Ben Elias, Columbia University Department of Mathematics,
  New York, NY 10027}
\email{belias@math.columbia.edu}

\author[Silberman]{Lior Silberman}
\address{Department of Mathematics, University of British Columbia,
Vancouver\ \ BC\ \ V6T 1Z4, Canada}
\email{lior@math.ubc.ca}

\author[Takloo-Bighash]{Ramin Takloo-Bighash}
\address{Ramin Takloo-Bighash, Department of Math, Stat, and Comp Sci,
University of Illinois at Chicago, Chicago, IL 60607}
\email{rtakloo@math.uic.edu}

\begin{abstract}
A minimal permutation representation of a finite group $G$ is a
faithful $G$-set with the smallest possible size. We study
the structure of such representations and show that for certain groups
they may be obtained by a greedy construction.  In these situations
(except when central involutions intervene) all minimal permutation
representations have the same set of orbit sizes. Using the same ideas
we also show that if the size $d(G)$ of a minimal faithful $G$-set is
at least $c|G|$ for some $c>0$ then $d(G) = |G|/m + O(1)$ for an integer $m$,
with the implied constant depending on $c$.
\end{abstract}

\maketitle

\section{Introduction}

It is a classical theorem of Cayley's that a group $G$ is isomorphic
to a subgroup of a symmetric group.  Accordingly we let the \emph{degree}
$d(G)$ of the finite group $G$ be the least integer $d$
such that $G$ can be embedded in $S_{d}$, the symmetric group on
$d$ letters. More precisely, Cayley's discussion in \cite{Cayley:PermRepn}
implicitly relies on the observation that the regular action of the
group on itself gives an embedding of $G$ into $S_{n}$, where $n=\card{G}$
is the order of $G$. It is then natural to ask to what extent the
resulting bound $d(G)\leq n$ is sharp.

The problem of finding $d(G)$ was first studied by Johnson
\cite{Johnson:MinPermRepn}. Among other things, he classified those
groups for which $d(G)=n$.  Except for a family of $2$-groups, these
groups are precisely the cyclic $p$-groups.  A structure theorem for
groups with $d(G)\geq c n$, $c$ any fixed positive constant, was
obtained by \cite{BabaiGoodmanPyber:ConstDelta} (see Remark
\ref{rem:BGP} below), while related results were obtained by
Berkovich in \cite{Berkovich:DegreeIndex}.

Although easy to define, the degree is difficult to compute. It is
more-or-less obvious that $d(G)$ can be computed by examining all
subsets of the subgroup lattice of $G$. The main conceptual finding of this note
is that in some cases a ``greedy'' algorithm is also available, that
is an algorithm that proceeds by making locally optimal choices rather
than directly searching for the global minimum.
This is hardly of practical application (the subgroup
lattice of a group may be exponentially larger than the group
itself), but it has surprising consequences for the structure of a
minimal permutation representation. We note that whenever a group
$G$ acts on a set $X$, the sizes of the orbits of the action
determine a partition of $\card{X}$. Our main application is:

\begin{thm}
\label{thm:mainthm}Let $G$ be a finite nilpotent group of odd order.
For each prime $p$, let $e_p$ be maximal such that the center of $G$
contains a subgroup isomorphic to the elementary abelian group $\FF_p^{e_p}$.
Let $X$ be a minimal faithful permutation representation of $G$. Then,
\begin{enumerate}
\item The number of orbits for the $G$-action on $X$ is $\sum_p e_p$;
\item The multiset of sizes of the orbits is a group isomorphism invariant.
\end{enumerate}
\end{thm}

This is a special case of a more general result, Theorem \ref{thm:mainthm2} below.
We remark that a restriction of the odd-order type is necessary, the
simplest counterexample being the four-group $C_2 \times C_2$.  Its
regular representation is a minimal permutation representation, but it
also has minimal representations with two orbits of size $2$. Though not strictly necessary for the proofs
of Theorems \ref{thm:mainthm} and \ref{thm:mainthm2}, we include Theorem \ref{thm:algo-works}. This theorem, which gives a method to find all \emph{perfect} minimal faithful permutation representations (c.f. Definition \ref{defn:perfect}), forms the conceptual backbone of our work.

The main motivation of this work was to understand the
distribution of $\Delta(G) \eqdef d(G) / \card{G}$ in the interval $[0, 1]$.
For example, it was easy to show that every number of the form $\frac{1}{n}$,
$n$ a positive integer, is a limit point of $\Delta(G)$ as $\card{G}$ tends
to infinity. Clearly, zero is also a limit point. We show here (see Theorem
\ref{thm:limitpts} below) that these are the only limit points.

This paper is organized as follows. In Section 2 we recall basic definitions. Section 3 contains our main results. Section 4 contains our study of limit points of $\Delta(G)$ in the interval $[0,1]$, plus some numerical results.

\section{\label{sec:indiv}Definitions}

We review some notation dealing with standard constructions of group actions.
For further details and basic definitions see e.g. sections 1.1-1.4 of \cite{Cameron:PermGps}, or sections 1.3-1.4 of \cite{dixonmortimer}. For basic materials on the socle see section 4.3 of \cite{dixonmortimer}.

Let $G$ be a finite group acting on a set $X$. We call this action a \emph{minimal faithful permutation representation} if the action is faithful, and the size of the set $X$ is the smallest possible among all sets on which $G$ acts in a faithful fashion. Under the action of $G$, the set $X$ decomposes as a disjoint union of orbits. Choosing a point stabilizer subgroup in each orbit, it is clear that minimal faithful permutation representations correspond to collections\footnote{We shall use the term ``colllection'' for such sets of subgroups.} $\HH$ of subgroups of $G$ where:

\begin{enumerate}
\item The \emph{core} of $\HH$,
$$ \CG{\HH} \eqdef \bigcap_{H\in\HH} \CG{H} =
 \bigcap_{H\in\HH} \bigcap_{g\in G} H^g $$
is trivial, and
\item $\sum_{H\in\HH}[G:H]$ is minimal among all $\HH$ satisfying (1).
\end{enumerate}

We call such sets $\HH$ ``minimal faithful collections''; they are
the subject of this paper.  The first condition corresponds to faithfulness
of the action, the second to the minimality of the degree.  Clearly if $\HH$
is a minimal faithful collection, no two of its elements can be
conjugate.

Note that the core of a subgroup $H < G$ is precisely the largest normal subgroup of $G$ contained in $H$.

We shall make use of the \emph{socle}, $\socG$, of a finite
group $G$, the subgroup generated by the set $\MM(G)$ of all
minimal normal subgroups of $G$.  Specifically, the lattice
$\TT(G) = \left\{ T\ideal G \mid T\subset \socG \right\}$ of normal
subgroups of $G$ contained in the socle will play a major role.

Every element $T\in\TT$ can be written as a direct product of
minimal normal subgroups (\cite{Suzuki:I} Thm II.4.8 p.131).  Moreover, the number of factors in any such direct product is an invariant of the pair $(G,T)$.  We denote it $\dim_G T$ and call it the \emph{dimension} of $T$.   In the language of order theory,
the lattice $\TT$ is \emph{atomic} with the minimal normal subgroups
being the atoms.  Since the lattice of normal subgroups
of $G$ is modular, both $\TT$ and its dual are \emph{matroids}. For readers unfamiliar with this theory, one should heuristically think of $\TT$ as behaving in a similar fashion to the lattice of subspaces of a vector space.

When $G$ is nilpotent every normal subgroup intersects the center (\cite{Suzuki:II} Thm IV.2.9 p. 18).
The discussion above is then elementary.  Since every subgroup of
the center is normal, $\socG = \soc(Z(G))$.  Furthermore, the socle
is a product of elementary abelian $p$-groups.

For a subgroup $H$ of $G$ we write $\RC{H}$ for the \emph{relative core}
of $H$, the subgroup $\CG{H}\cap \socG$.  It is then clear that
$$ \RC{H} = \left< N\in\MM(G) \mid N\subset H \right>. $$
For a collection $\HH$ of subgroups we similarly set
$$ \RC{\HH} \eqdef \bigcap_{H\in\HH} \RC{H} = \CG{\HH}\cap \socG. $$
It is clear that $\CG{\HH}$ is trivial if and only if $\RC{\HH}$ is trivial.
This simple observation underlies our later analysis.

We also occasionally write $\HH_M$ for $\RC{\HH}$, and $H_M$ for $\RC{H}$.

We extend the notion of dimension above to all subgroups of $G$ by setting
$\dim_G(H) = \dim_G(\RC{H})$.  In particular we write $\dim G$ for
$\dim_G(G) = \dim_G(\socG)$.  We will also use the \emph{codimension}
$\codim_G(H) = \dim G - \dim_G(H)$.

\section{\label{sec:algo}Determining $d(G)$}

We discuss here the (algorithmic) problem of constructing a minimal
permutation representation of $G$.  As input, we give ourselves the
subgroup lattice of $G$ and, in addition, the order of each subgroup
and whether it is normal in $G$ or not.  This analysis will shed light
on the structure of the minimal permutation representations.

\subsection{A special class of groups}

\begin{defn}\label{defn:socle-friendly}
Let $G$ be an arbitrary finite group, and let $\TT$ be as above. We call $G$
\emph{socle friendly} if for all $H\subgp G$, $T\in\mathcal{T}$, we have $\RC{H\cdot T}=\RC{H}\cdot T$.
\end{defn}

\begin{lem}
\label{lem:enlarge-kernel} If $G$ is a nilpotent group, then $G$ is socle friendly.
\end{lem}
\begin{proof}
Since the lattice $\TT$ is relatively complemented, we may
write $T=(T\cap \RC{H})\cdot S$ for some $T\in S$ disjoint to $\RC{H}$.
We then have $H\cdot T=H\cdot S$ and $\RC{H}\cdot T=\RC{H}\cdot S$ so we may assume
$H\cap T=\left\{ 1\right\}$. Clearly $\RC{H}\cdot T\subset\RC{H\cdot T}$.
Conversely, let
$N\subgp HT$ be a minimal normal subgroup of $G$. If $N\subgp T$
there is nothing to prove, so we may assume $T\cap N=\left\{
1\right\}$.  Since $H$ and $T$ are disjoint, every $n\in N$ can be
uniquely written in the form $n = h_n t_n$ for some $h_n \in H$ and
$t_n \in T$.  Note that the map $n\mapsto h_n$ is a group homomorphism
(it is the restriction to $N$ of the quotient map $H\cdot T/T\isom H$), and since
$N$ and $T$ are disjoint it is an isomorphism onto its image $N'$.

Since $N$ and $T$ are central subgroups (here we use the nilpotence of
$G$), it follows that $N'$ is a central subgroup
as well, and since $N$ was a cyclic group of prime order so is $N'$.
It follows that $N'$ is a minimal normal subgroup of $G$, contained in $H$.
We conclude that $N \subset N' T \subset \RC{H}\cdot T$.
\end{proof}

\begin{rem}\label{saunders}
Not every finite group is socle friendly. Here is the construction of an infinite family of examples simplifying the construction of \cite{saunders}. Let $H$ be any finite group with two non-isomorphic one dimensional representations $V_1, V_2$ over a finite field $\mathbb{F}$. We let $V = V_1 \oplus V_2$ and $G= H \rtimes V$. Then $\soc(G) = V$ and $\TT = \{ 0, V_1, V_2, V\}$. Let $W$ be
any one dimensional $\mathbb{F}$-subspace not containing either of $V_1, V_2$. Then $W$ is core-free and consequently $\RC{W} \cdot V_1 = V_1$ and $\RC{W} \cdot V_2 = V_2$. But $W\cdot V_1 = W \cdot V_2 = V$ and as a result $\RC{W \cdot V_1} = \RC{W \cdot V_2} = V$. This shows that $G$ is not socle friendly.
\end{rem}

\subsection{Minimal faithful collections and codimension one subgroups}

Let $G$ be a finite socle friendly group.  We are interested in
constructing a minimal faithful collection of subgroups of $G$, and
a natural way to do so is step-by-step, incrementally adding
subgroups to our collection until it is faithful. Rather than
keeping track of $\CG{\HH}$, we note that $\RC{\HH}$
carries sufficient information to decide whether $\CG{\HH}$ is
trivial. Moreover, while the cores $\CG{\HH}$ decrease through
the lattice of all normal subgroups of $G$, the relative cores
$\RC{\HH}$ decrease through the lattice $\mathcal{T}(G)$
which is much easier to work with.

We now turn to the ``minimality'' property of a collection, which
appears to push in the opposite direction to ``faithfulness''. The
first favors selecting large subgroups, and having few of them. The
second seems to suggest choosing small subgroups, or else many large
ones will be needed. The multiplicative property of orders of
subgroups actually implies that choosing many large subgroups is the right way.
The analysis is very similar to that of Johnson \cite{Johnson:MinPermRepn}.
In both cases it is shown that the elements of a minimal faithful
collection may be (and in some cases, must be) drawn from a
particular class of subgroups, using the same trick.  The reader should
compare the following Lemma with \cite[Lemma 1]{Johnson:MinPermRepn}

\begin{lem}
\label{lem:replace}(``replacement lemma'')
Let $H\subgp G$ be of codimension at least $2$.
Then there exist subgroups $H_1$ and $H_2$ of $G$ containing $H$
such that $\RC{H_1}\cap \RC{H_2} = \RC{H}$ and
$\frac{1}{\card{H_{1}}}+\frac{1}{\card{H_{2}}}\leq\frac{1}{\card{H}}$.
Moreover, this inequality is strict unless $G$ contains at least two
central involutions.
\end{lem}
\begin{proof}
Since $\TT$ is a matroid and $\RC{H}$ has codimension at least $2$,
there exists two minimal normal subgroups $N_{1}, N_{2} \in \MM(G)$
(``atoms of the lattice $\TT(G)$'') such that
the lattice join $\RC{H}N_{1}N_{2}$ has dimension greater by $2$ than that
of $\RC{H}$.  In other words, that lattice join is a direct product.
The inclusions $\RC{H}\subgp \RC{H}N_{i}$ are then proper, and we have
$\RC{H} = \RC{H}N_{1} \cap \RC{H}N_{2}$.

We thus set $H_{i}=H\cdot N_{i}$, $i=1,2$
(these are semi-direct products as the $N_{i}$ are minimal normal subgroups).
By Lemma \ref{lem:enlarge-kernel}, $\RC{H_{i}} = \RC{H}N_{i}$, and
it follows that $\RC{H_{1}} \cap \RC{H_{2}} = \RC{H}$.
Since $H$ is a proper subgroup of both
$H_{1},H_{2}$ its index in both subgroups is at least $2$, and we have
$$
\frac{1}{\card{H_{1}}}+\frac{1}{\card{H_{2}}}\leq\left(\frac{1}{2}+\frac{1}{2}\right)\frac{1}{\card{H}}=\frac{1}{\card{H}}.
$$

Equality can only happen if both $N_{1}$ and $N_{2}$ are of order
$2$, in which case the non-trivial elements of $N_{i}$ are both
central involutions.
\end{proof}
\begin{defn}
Let $\CA=\CA(G)$ denote the set of subgroups of $G$
of codimension $1$.
\end{defn}

The reader should compare the next theorem with
\cite[Cor.\, 1]{Johnson:MinPermRepn}.

\begin{thm}\label{thm:codim1-suffice}
There exist minimal faithful collections contained in $\CA$,
and these are the ones of maximal size.  If $G$ has at most one
central involution then every minimal faithful collection is
contained in $\CA$.
\end{thm}
\begin{proof}
Let $\HH$ be a faithful collection, and let $H\in\HH$.
If $H$ is of codimension $0$ (\ie $\RC{H}=\socG$) we have
$$
\left\{ 1\right\} = \RC{\HH} =
\RCore\left(\HH\setminus\left\{ H\right\} \right) \cap \RC{H} =
\RCore \left(\HH\setminus\left\{ H\right\} \right).
$$
In particular, $\mathcal{H}\setminus\left\{ H\right\} $ is also
faithful. If $H$ has codimension at least $2$, let $H_{1},H_{2}$
be the subgroups constructed in Lemma \ref{lem:replace},
and let
$\HH'=\left(\HH\setminus\left\{ H\right\}
\right)\cup\left\{ H_{1},H_{2}\right\} $. By construction we have
$\RC{\HH'} = \RC{\HH} = \left\{ 1\right\} $ so that
$\mathcal{H}'$ is faithful. In addition, Lemma \ref{lem:replace}
yields $\Delta(\HH')\leq\Delta(\HH)$, with a strict inequality
if $G$ has at most one central involution.  In
general we note that $\HH'$ has more elements than
$\HH$. In particular, a minimal faithful collection of
maximal size must consist of codimension-one subgroups.
\end{proof}
\begin{defn}
Call a collection $\HH\subset\CA$
\emph{independent} if its relative core is strictly contained in
that of any proper sub-collection; or, in other words, if $\{\RC{H} \mid H\in\HH\}$
is an independent set of atoms in the lattice dual to $\TT$.
\end{defn}
A minimal faithful collection $\mathcal{H}\subset\CA$ is
certainly independent -- otherwise it would have a faithful proper
sub-collection.

\begin{prop}
\label{pro:matroid}The set of independent collections of $\CA$
forms a matroid, i.e. the following statements are true:
\begin{enumerate}
\item A subcollection of an independent collection is independent.
\item $\mathcal{H}\subset A$ is independent if and only if $\codim_{G}\mathcal{H}_{M}=\card{\mathcal{H}}$.
\item \label{enu:mat-replace}If $\mathcal{H}$, $\mathcal{H}'$ are independent
collections with $\card{\mathcal{H}'}>\card{\mathcal{H}}$ then there
exists $H'\in\mathcal{H}'$ such that $\mathcal{H}\cup\left\{ H'
\right\} $ is independent.
\end{enumerate}
\end{prop}
\begin{proof}
This will follow via the replacement Lemma from the general fact that $\TT$
is a matroid.
\begin{enumerate}
\item Let $\mathcal{H}\subset\CA$ be independent, and suppose  $\mathcal{H}''$ is a proper subcollection
of $\mathcal{H}'\subset\mathcal{H}$ such that
$\mathcal{H}''_{M}=\mathcal{H}'_{M}$. Letting
$\bar{\mathcal{H}}=\mathcal{H}\setminus\mathcal{H}'$, we have \[
\left(\bar{\mathcal{H}}\cup\mathcal{H}''\right)_{M}=\bar{\mathcal{H}}_{M}\cap\mathcal{H}''_{M}=\bar{\mathcal{H}}_{M}\cap\mathcal{H}'_{M}=\mathcal{H}_{M},\]
contradicting the independence of $\mathcal{H}$.
\item Let $S,T\in\mathcal{T}(G)$ with $\codim_{G}T=1$. Then $ST$ either
equals $T$ or $M$, and we have $\dim_{G}S\cap T=\dim_{G}S$ or
$\dim_{G}S-1$, respectively, by the inclusion-exclusion formula for dimension. By induction on
the size of any collection
$\mathcal{H} = \left\{H^i\right\}_{i=1}^{k}\subset\CA$ we see
that $\codim_{G}\mathcal{H}_{M}\leq\card{\mathcal{H}}$, with
equality if and only if the sequence of intersections $\cap_{i=1}^{m} \RC{H^i}$
is strictly decreasing with $m$, $1\leq m \leq k$.
\item We have $\dim_{G}\mathcal{H}'_{M}<\dim_{G}\mathcal{H}_{M}$, and hence
$\mathcal{H}'_{M}$ does not contain $\mathcal{H}_{M}$. It follows
that we can find $H'\in\mathcal{H}'$ such that $H'_{M}$ does not
contain $\mathcal{H}_{M}$. Then
$\dim_{G}(\mathcal{H}_{M}\cap H'_{M})=\dim_{G}\mathcal{H}_{M}-1$ (equality
is not possible by the choice of $H'$). By part (2) we see that
that $\mathcal{H}\cup\left\{ H'\right\} $ is independent.
\end{enumerate}
\end{proof}
\begin{cor}
Let $\mathcal{H}\subset A$ be independent. Then the following are
equivalent:
\begin{enumerate}
\item $\card{\mathcal{H}}=\dim G$;
\item $\mathcal{H}$ is faithful;
\item $\mathcal{H}$ is a maximal independent subset of
$\CA$. Here, \emph{maximal} means maximal with respect to
inclusion.
\end{enumerate}
\end{cor}
\begin{proof}
The equivalence of (1) and (2) is contained in part (2) of
Proposition \ref{pro:matroid}. An independent collection with
$\mathcal{H}_{M}=\left\{ 1\right\} $ is certainly maximal. An
independent collection with $\mathcal{H}_{M}\neq\left\{ 1\right\} $
is not maximal since in that case there exists some
$T\in\mathcal{T}$ of codimension $1$ which does not contain
$\mathcal{H}_{M}$, and we can add it to $\mathcal{H}$ to form a
larger independent collection.
\end{proof}
\begin{cor}
A subset $\mathcal{H}\subset\CA$ is a minimal faithful
collection if and only if it is independent and maximizes \[
w(\mathcal{H})=\sum_{H\in\mathcal{H}}\left(2-\frac{1}{\card{H}}\right)\]
among the independent subsets.
\end{cor}
\begin{proof}
We have already noted that a minimal faithful collection contained
in $\CA$ is independent and maximal (with respect to
inclusion), and that a maximal (with respect to inclusion)
independent set is a faithful collection. It is clear that a subset
maximizing this weight function is maximal independent, since
$2-\frac{1}{\card{H}}>0$ for all subgroups $H$. Finally, we note
that a maximal independent set H satisfies
\[ w(\mathcal{H})=2\dim
G-\Delta(\mathcal{H}).\]

\end{proof}
\begin{cor}
\label{cor:card-faith}There exist minimal faithful collections of
size $\dim G$. If $G$ has more than one central involution,
there may also exist minimal faithful collections of smaller
size.
\end{cor}
\begin{proof}
We have seen that there exist minimal faithful collections contained in
$\CA$, that these are independent sets, and that every independent
set has $\dim G$ elements.
\end{proof}

Inspired by this Corollary we make the following definition:
\begin{defn}\label{defn:perfect}
  A minimal faithful collection of size $\dim G$ is called \emph{perfect}. Correspondingly, a minimal
  faithful permutation representation with $\dim G$ orbits under the $G$-action is called \emph{perfect}.
\end{defn}

\begin{example}
Let $G$ be a $p$-group for a prime $p$, and let $Z=Z(G)$ be its
center. It is well-known (and follows from the class formula) that
every normal subgroup of $G$ intersects the center non-trivially.
Since every subgroup of the center is normal, it follows that
$\mathcal{M}(G)=\mathcal{M}(Z)$, and in particular $\dim G=\dim
Z(G)$. This observation recovers \cite[Thm. 3]{Johnson:MinPermRepn}:
\end{example}
\begin{thm}
Let $G$ be a $p$-group with center $Z$. Then there exists a minimal
faithful collection for $G$ of size $\dim Z$. If $p$ is odd
this holds for all minimal faithful collections.
\end{thm}

\subsection{\label{sub:algo}Construction}

In the remainder of this section we assume that $G$ is a socle friendly finite group.  We have reduced the problem of finding a minimal faithful collection to maximizing an additive weight function on a matroid. This is a problem which is solvable by a greedy algorithm, and thus we may search for and construct perfect minimal faithful permutation representations. Before we present our method we record a useful Lemma:

\begin{lem}\label{lem:simplifying}
Let $\mathcal{H}\subset\CA$ be independent, and suppose
$H'\subgp G$ has the largest size possible such that $H'_{M}$
does not contain $\mathcal{H}_{M}$. Then $H'\in\CA$,
$\mathcal{H}\cup\left\{ H'\right\} $ is independent, and $H'$
maximizes the function $w(H)=2-\frac{1}{\card{H}}$ among all
$H\in\CA$ such that $\mathcal{H}\cup\left\{ H\right\} $ is
independent.
\end{lem}
\begin{proof}
We can find
$T\in\mathcal{T}$ of codimension $1$ containing $H'_{M}$ but not
containing $\mathcal{H}_{M}$. Setting $H=H'T$ we have $H_{M}=H'_{M}T=T$,
which does not contain $\mathcal{H}_{M}$. By the maximality of $H'$
we have $H=H'$ implying $H'_{M}=T$, so that $H'$ is of codimension
$1$ and $\mathcal{H}\cup\left\{ H'\right\} $ is independent. Finally
$H'$ was chosen to maximize $w(H)$ in an even larger family than
needed.
\end{proof}

We now describe a method to find all perfect minimal faithful permutation representations.
We assume we are given the following data:

\begin{enumerate} \item The subgroup lattice of $G$;
\item the sizes of every element of the subgroup lattice;
\item and that normal subgroups are marked as such.
\end{enumerate}

Then for each $i \geq 0$ we recursively construct a sequence of triples $(\HH_i, T_i, \Delta_i)$ with
each $\HH_i$ a collection of subgroups of $G$, $T_i$ a subgroup of $G$, and $\Delta_i$ a non-negative real number. In order to do this we proceed as follows. Let $\mathcal{H}_0 = \emptyset$, $T_0=\socG$, $\Delta_0=0$. Now suppose
$(\HH_i, T_i, \Delta_i)$ is given, and $T_i \ne \{1\}$. First we find a subgroup $H_{i+1}$ of $G$ of maximal size not containing $T_i$. Then we set $\HH_{i+1} = \HH_i \cup \{ H_{i+1}\}$, $T_{i+1} = T_i \cap \CG{H_{i+1}}$, $\Delta_{i+1} = \Delta_i + \frac{1}{\card{H_{i+1}}}$. If $T_i = \{1\}$, we simply set $(\HH_{i+1}, T_{i+1}, \Delta_{i+1}) = (\HH_i, T_i, \Delta_i)$.
The sequence $(\HH_i, T_i, \Delta_i)$ is certainly not unique and depends on the choices of the subgroups $H_i$.
Then we have the following theorem:

\begin{thm} Let $G$ be a socle friendly finite group, and let $\dim G = \delta$. Then
\label{thm:algo-works}
\begin{enumerate} \item For any choice of the subgroups $H_i$, $T_{\delta-1} \ne \{1 \}$ whereas $T_\delta = \{1\}$. Furthermore, $\HH_\delta$ is a minimal faithful collection of size $\delta$, and $\Delta_\delta = \Delta(G)$.
\item Conversely, up to $G$-isomorphism any minimal faithful collection of size $\delta$ can be obtained this way.
\end{enumerate}
\end{thm}
\begin{proof}
First we prove the first part. From Lemma \ref{lem:simplifying} it is clear for each $i$ the collection $\HH_i$ is independent, and $T_i = (\HH_i)_M$. Also it is easy to see that for each $i$, $\dim T_{i+1} = \dim T_i -1$ as long as $T_i \ne \{1\}$.
These observations immediately give the first assertion of the theorem.
We show by induction that for $k \leq \delta$,
$\sum_{H\in\HH_k}\frac{1}{\card{H}}$ is minimal among
independent collections of size $k$. This is certainly the case
for $k=0$. Thus let $\HH_{k-1}$ be given, and choose the subgroup $H_k$. Suppose there
is an independent collection $\mathcal{H}'\subset\CA$ of
size $k$ such that
$\sum_{H'\in\mathcal{H}'}\frac{1}{\card{H'}}<\frac{1}{\card{H_{k}}}+\sum_{H\in\HH_{k-1}}\frac{1}{\card{H}}$.
We may then write $\mathcal{H}'=\mathcal{H}''\cup\left\{
H'_{k}\right\} $ where $H'_{k}$ is a member of minimal size.
By the inductive hypothesis,
$\sum_{H\in\HH_{k-1}}\frac{1}{\card{H}}\leq\sum_{H'\in\mathcal{H}''}\frac{1}{\card{H'}}$,
and hence we must have $\card{H_{k}} < \card{H'_{k}}$. By the choice
of $H'_{k}$, we actually have $\card{H_{k}}<\card{H'}$ for all
$H'\in\mathcal{H}'$. We now use the matroid property of the
independent subcollections of $\CA$ shown in Proposition
\ref{pro:matroid}(\ref{enu:mat-replace}): since $\mathcal{H}'$ is of
size $k$, while $\mathcal{H}_{k-1}$ is of size $k-1$, there exists some
$H'\in\mathcal{H}'$ such that $\mathcal{H}_{k-1}\cup\left\{ H'\right\} $
is independent. In particular this implies that
$\left(\mathcal{H}_{k-1}\cup\left\{ H'\right\} \right)_{M}$ is strictly
contained in $\mathcal{H}_{M}$, and as $\card{H'}>\card{H_{k}}$ we
have a contradiction to the existence of $\mathcal{H}'$.

Now we prove the second part.  Let $\mathcal{H}=\left\{ H_{i}\right\} _{i=1}^{\delta}$ be a minimal faithful
collection, ordered such that \[
\card{H_{1}}\geq\card{H_{2}}\geq\cdots\geq\card{H_{\delta}}.\] Then
we claim that each $k$, $H_{k}$ has maximal size among all subgroups $H'$ of $G$
such that $\left(\left\{ H_{i}\right\} _{i=1}^{k-1}\cup\left\{
H\right\} \right)_{M}$ is a proper subgroup of $\left(\left\{
H_{i}\right\} _{i=1}^{k-1}\right)_{M}$.
By induction, it suffices to check that if a subgroup $H'<G$
is independent of $\left\{H_i\right\}_{i=1}^{k-1}$ then
there exists $l\geq k$ such that
$\mathcal{H} \cup \left\{H'\right\} \setminus \left\{H_l\right\}$
is independent.  For this we set $S_j = \cap_{i=1}^{j} \RC{H_i}$.
It is then easy to see that we may take $l$ to be the first $j$
such that $\RC{H'}\cap S_j = S_j$. The assertion of the theorem is now immediate.
\end{proof}

\subsection{The main theorem}

In this section we state and prove our main theorem. We start with a definition:
\begin{defn}
   Let $G$ be a finite group. Given a permutation representation $X$ we denote by $m(X)$ the multi-set consisting of the sizes of the orbits of $X$ under the $G$-action.
\end{defn}

\begin{thm}
\label{thm:mainthm2}Let $G$ be a socle friendly finite group.
Let $X$ be a minimal faithful permutation representation of $G$. Then,
\begin{enumerate}
\item The number of orbits of $X$ under the action of $G$ is at most $\dim G$;
\item $G$ has perfect minimal faithful permutation representations; and
if the center of $G$ has at most one involution then every faithful permutation representation is perfect;
\item If $X_1$, $X_2$ are two perfect minimal faithful representations of $G$, then $m(X_1) = m(X_2)$.
\end{enumerate}
\end{thm}

\begin{proof}
  The first two parts of the theorem follow from Corollary \ref{cor:card-faith}. The third part easily follows from
  Theorem \ref{thm:algo-works} and its proof.
\end{proof}

\section{Applications}

\subsection{Accumulation points of $\Delta(G)$}

Let $n,p\in \N$ with $p>n$ a prime.  Then $\Delta(C_n \times C_p) =
\frac{1}{n} + \frac{\Delta(C_n)}{p} = \frac{1}{n} + O(\frac{1}{p})$.
In particular, $\lim_{p\to\infty} \Delta(C_n\times C_p) =
\frac{1}{n}$. This means that for each positive integer $n$, the
point $\frac{1}{n}$ is an accumulation point of the set $\{
\Delta(G); $G$ \text{ finite group}\}$ in the interval $[0, 1]$. In Theorem \ref{thm:limitpts} below we show that these points
are the only non-zero accumulation points. We begin with some
preliminary lemmas.

\begin{lem}
\label{lem:ind-rep}Let $H\subgp G$ be a subgroup. Then $d(H)\leq
d(G)$ and $\Delta(G)\leq\Delta(H)$.
\end{lem}
\begin{proof}
The first claim is obvious. For the second, let $\mathcal{H}'$ be a
faithful collection of subgroups of $H$ and note that
$\Delta(\mathcal{H})$ is independent of the ambient group. Then
$K_{G}(H_{i})\subset K_{H}(H_{i})$ (larger intersection). In
particular, $K_{G}(\mathcal{H})=\left\{ 1\right\} $. Choosing
$\mathcal{H}$ minimal for $H$ we deduce that
$\Delta(G)\leq\Delta(\mathcal{H})=\Delta(H)$.
\end{proof}
\begin{rem}
\label{rem:BGP}A cyclic $p$-group has relative degree $1$. In
particular, if $P\subgp G$ is a cyclic $p$-group then \[
\Delta(G)\geq\frac{d(P)}{\card{G}}=\frac{1}{[G:P]}.\] Conversely,
Babai-Goodman-Pyber \cite{BabaiGoodmanPyber:ConstDelta} give an
explicit function $f\colon[0,1]\to\R$ such that if
$\Delta(G)\geq\Delta$ then $G$ has a cyclic $p$-subgroup of index at
most $f(\Delta)$. In other words, as $\card{G}$ grows with
$\Delta(G)\geq\Delta$, the degree of $G$ is controlled (up to
bounded multiplicative error) by the size of the largest cyclic
$p$-subgroup of $G$.  Specifically, they show that when $G$ does not
possess a large cyclic group of prime-power order it has a pair of
reasonably large subgroups with trivial intersection.

Note that the above bound on $\Delta(G)$ is derived from a faithful
collection of size $2$.  In Lemma \ref{lem:small-dim} we show that
when $\Delta(G) \geq \Delta$ there exists $k$ depending only on
$\Delta$ such that a minimal permutation representation of $G$ has
at most $k$ orbits. The case of groups of prime exponent and
nilpotence class two, studied in \cite[Thm.\
3.6]{BabaiGoodmanPyber:ConstDelta} as well as
\cite{Neumann:PermGpAlg}, shows that we need $k > 2$ in general.
\end{rem}

\begin{lem}\label{lem:small-dim} Let $k = \dim G$. Then
$\Delta(G) \leq \frac{k}{2^{k-1}}$.
\end{lem}
\begin{proof}
Write the socle $M=\socG$ as the direct product
of $k$ minimal normal subgroups $\{S_i\}_{i=1}^k$.  For $1 \leq
i\leq k$ let $H_i = \prod_{j\neq i} S_j$. It is clear that $\{H_i\}$
is a faithful collection of size $k$ and each of its elements has
size at least $2^{k-1}$.
\end{proof}

\begin{lem}\label{lem:cyc-p-gp} Let $P$ be a cyclic $p$-subgroup of $G$.
Then $\RC{P} \subgp \soc(P)$.
If $|G|$ is large enough compared to $[G:P]$ then equality holds.
\end{lem}
\begin{proof}
Let $N<P$ be non-trivial and normal in $G$.  Then $\soc(P)$ is a
characteristic subgroup of $N$. It follows that $\RC{P}$ is either trivial
 or equal to $\soc(P)$. In any case, we have $\dim_G P\leq 1$.

Finally, the core of $P$ has index at most $([G:P])!$ (it is the
kernel of a homomorphism into $S_{[G:P]}$).  If $|G|> ([G : P])!$
then $\CG{P}$ is a non-trivial normal subgroup of $G$ contained in
$P$, hence containing its unique subgroup of order $p$.  In that case
$\soc(P)$ is normal in $G$ and thus $\RC{P} = \soc(P)$.
\end{proof}
In fact, if $G$ has a large cyclic $p$-subgroup then a permutation
representation with two orbits is almost optimal:

\begin{cor}\label{cor:ell} Let $P$ be a cyclic $p$-subgroup of $G$,
and let $l(G)$ be the order of the smallest point stabilizer in an
orbit in a minimal permutation representation of $G$. Then
$$ \frac{1}{l(G)} \leq \Delta(G) \leq \frac{1}{l(G)} + \frac{1}{|P|}.$$
\end{cor}
\begin{proof}
Let $\mathcal{H}$ be a minimal faithful collection for $G$, chosen
so that it contains an element $H_1$ of smallest possible order
(denoted above by $l(G)$). Clearly $\Delta(G) = \Delta(\mathcal{H})
\geq \frac{1}{l(G)}$. For the other assertion, we may as well assume $M(P) \in
\mathcal{M}(G)$, otherwise $\CG{P}=\{1\}$ and the claim is clear.
Then ${\mathcal H}$, being faithful, must contain an element $H_2$ disjoint from $M(P)$, hence $\{P,H_2\}$ is a
faithful collection.
\end{proof}

\begin{thm}\label{thm:limitpts}
Let $G_n$ be a sequence of groups with orders increasing to infinity
such that $\lim_{n\to\infty}\Delta(G_n)>0$. Then this limit is of
the form $1/l$ for some $l\in\N$.
\end{thm}
\begin{proof}
For $n$ large enough we have $\Delta(G_n) > \Delta >0$. The main
result of \cite{BabaiGoodmanPyber:ConstDelta}, already quoted above,
is that $G_n$ has a cyclic $p_n$-subgroup $P_n$ of index at most
$f(\Delta)$ for some $f\colon[0,1]\to\N$.  It follows that
$$ \left| \Delta(G_n) - \frac{1}{l(G_n)} \right| \leq
\frac{f(\Delta)}{\card{G_n}}.$$

Here $l(G_n)$ is as in the statement of Corollary \ref{cor:ell}.
As $\card{G_n}\to\infty$, we see that $\frac{1}{l(G_n)}$ tends to a
positive limit.  The sequence of integers $l(G_n)$ must then be
eventually constant, equal to an integer $l$. Corollary \ref{cor:ell}, combined with the fact that the size of $P_n$ goes to infinity, implies that $\lim_{n\to\infty}\Delta(G_n)=\frac{1}{l}$.
\end{proof}

Note that we have shown more, that if $\Delta(G) \geq \Delta > 0$
then any minimal permutation representation consists of one large
orbit of size essentially $\card{G}\Delta(G)$, and several other
orbits of size and number bounded in terms of $\Delta$.  Indeed, the
number of orbits is bounded by Lemma \ref{lem:small-dim}.
We have an obvious bound
$l(G)\leq\left(\Delta(G)-f(\Delta)/\card{G}\right)^{-1}$.
Next, as soon as $\card{G}$ is large enough so that
$\frac{1}{l(G)+1} + \frac{f(\Delta)}{\card{G}} < \frac{1}{l(G)}$,
the subgroups $H_1, H_2$ of Lemma
\ref{lem:cyc-p-gp} must have the same size.  We conclude that
if $\Delta(G) > \Delta$ and $\card{G}$ is large enough (depending on
$\Delta$), $G$ has a cyclic $p$-subgroup $P$ of index at most
$f(\Delta)$ such that $M(P)$ is normal in $G$ and a subgroup $H$ of
order $l(G)$ belonging to a minimal faithful collection and disjoint
from $M(P)$.  Then every other member of that minimal faithful
collection may be replaced with $P$ keeping the collection faithful.
Hence all other orbits in the representation must have size at most
$f(\Delta)$.

\subsection{Some numerical results}

The thesis \cite{BenThesis} contains an implementation of
procedure preceding Theorem \ref{thm:algo-works} in the algebraic programming language
\MAGMA \cite{MAGMAhome}. Using the limited computing power of a personal
computer, $p$-groups of order $p^n$ for $n \leq 6$ and small $p$
were examined. Any such group can be found in the \MAGMA  database.
Let us summarize the findings.

There is only one group $G$ of order $p$, and for this group
$\Delta(G) =1$. There are two groups of order $p^2$, namely $\Z_p
\times \Z_p$ and $\Z_{p^2}$. Here $\Delta(\Z_p \times \Z_p) =
\frac{2}{p}$ and $\Delta(\Z_{p^2}) = 1$. Consequently
$\sum_{|G|=p^2} \Delta(G) = 1 + \frac{2}{p}$. There are five groups
of order $p^3$: one cyclic with $\Delta=1$; one elementary abelian
with $\Delta=\frac{3}{p^{2}}$; one abelian with a generator of order
$p^{2}$, having $\Delta=\frac{1}{p}+\frac{1}{p^{2}}$; and two
non-abelian groups both having $\Delta=\frac{1}{p}$. Observe that
$\sum_{|G|=p^3} \Delta(G)=1+\frac{3}{p}+\frac{4}{p^{2}}$. For groups
of order $p^4$ and $p^5$ we state the following conjecture:
\begin{conj}
For $p>3$
\[
\sum_{|G|=p^{4}}\Delta(G)=1+\frac{5}{p}+\frac{11}{p^{2}}+\frac{9}{p^{3}},\]
\[
\sum_{|G|=p^{5}}\Delta(G)=1+\frac{7}{p}+\frac{34+2\gcd(p-1,3)+
\gcd(p-1,4)}{p^{2}}+\frac{54}{p^{3}}+\frac{24}{p^{4}}.\]
\end{conj}
For any prime $p\ge3$, there are exactly fifteen groups of order
$p^{4}$, and these can be enumerated and described. So the proof of
the first part of the conjecture should be straightforward. We have
computationally verified the conjecture for groups of order $p^4$ for every
prime $p$ in the range $3 < p < 50$ and several larger values of $p$ ($\approx
1000$). We considered the groups of order $p^5$ for $p \leq 19$. Note
that the number of groups of order $p^5$ is
$61+2p+2\gcd(p-1,3)+\gcd(p-1,4)$. For groups of
order $p^6$, we did not have enough data points to be able to guess
a formula.

\section*{Acknowledgments}

We would like to acknowledge conversations with John Conway, William
Kantor, Avinoam Mann and James Wilson.  Our interest in minimal
permutation representations was triggered by a question raised by
Andre Kornell in a Princeton undergraduate Algebra class.  During
the initial investigation, the third author was assisted by Evan
Hass, another student in that class. Theorem \ref{thm:algo-works} was initially
obtained for $p$-groups by the first author under the supervision of the third
author (\cite{BenThesis}).  In an earlier draft of this paper we had claimed that every
finite group is socle friendly (Definition \ref{defn:socle-friendly}). We wish to thank Neil Saunders
for pointing out this inaccuracy and constructing a counterexample (\cite{saunders}; also c.f. Remark \ref{saunders} where a simplification of Saunders' example is presented.). We also thank two anonymous referees for their careful reading of the manuscript and their many
suggestions which led to considerable improvement in the style and
presentation of the paper.

The second author's research was supported in part by a Porter Ogden
Jacobus Fellowship at Princeton University, a Clay Mathematics Foundation
Liftoff Fellowship, and the National Science Foundation under agreement
No. DMS-0111298. The third author's research was partially funded by the Young Investigator
Grant No. 215-6406 from the NSA and by the National Science Foundation grant No. DMS-0701753.

Any opinions, findings and conclusions or recommendations expressed
in this material are those of the authors and do not necessarily
reflect the views of the NSF and the NSA.

\providecommand{\bysame}{\leavevmode\hbox to3em{\hrulefill}\thinspace}
\providecommand{\MR}{\relax\ifhmode\unskip\space\fi MR }
\providecommand{\MRhref}[2]{%
  \href{http://www.ams.org/mathscinet-getitem?mr=#1}{#2}
}
\providecommand{\href}[2]{#2}


\begin{thebibliography}{10}

\bibitem{BabaiGoodmanPyber:ConstDelta}
L{\'a}szl{\'o} Babai, Albert~J. Goodman, and L{\'a}szl{\'o} Pyber, \emph{On
  faithful permutation representations of small degree}, Comm. Algebra
  \textbf{21} (1993), no.~5, 1587--1602.

\bibitem{Berkovich:DegreeIndex}
Yakov Berkovich, \emph{The degree and index of a finite group}, J. Algebra
  \textbf{214} (1999), no.~2, 740--761.

\bibitem{Cameron:PermGps}
Peter~J. Cameron, \emph{Permutation groups}, London Mathematical Society
  Student Texts, vol.~45, Cambridge University Press, Cambridge, 1999.

\bibitem{Cayley:PermRepn}
Arthur Cayley, \emph{Desiderate and suggestions: No. 1. the theory of groups},
  Amer. J. Math. \textbf{1} (1878), no.~1, 50--52.

\bibitem{dixonmortimer}
John~D. Dixon and Brian Mortimer, \emph{Permutation groups}, Graduate Texts in
  Mathematics, vol. 163, Springer-Verlag, New York, 1996.

\bibitem{BenThesis}
Benjamin Elias, \emph{Minimally faithful group actions and $p$-groups}, 2005,
  Princeton University Senior Thesis.

\bibitem{Johnson:MinPermRepn}
D.~L. Johnson, \emph{Minimal permutation representations of finite groups},
  Amer. J. Math. \textbf{93} (1971), 857--866.

\bibitem{Neumann:PermGpAlg}
Peter~M. Neumann, \emph{Some algorithms for computing with finite permutation
  groups}, Proceedings of groups---St.\ Andrews 1985 (Cambridge), London Math.
  Soc. Lecture Note Ser., vol. 121, Cambridge Univ. Press, 1986, pp.~59--92.

\bibitem{saunders}
Neil Saunders, \emph{Private communication}, 09/01/2007.

\bibitem{Suzuki:I}
Michio Suzuki, \emph{Group theory. {I}}, Grundlehren der Mathematischen
  Wissenschaften, vol. 247, Springer-Verlag, Berlin, 1982.

\bibitem{Suzuki:II}
\bysame, \emph{Group theory. {II}}, Grundlehren der Mathematischen
  Wissenschaften, vol. 248, Springer-Verlag, New York, 1986.

\bibitem{MAGMAhome}
MAGMA Computational~Algebra System, \emph{http://magma.maths.usyd.edu.au/}.

\end{thebibliography}
\end{document}